\documentclass{amsart}
\usepackage{graphicx}
\usepackage{amsthm,amsfonts,amssymb,amsmath}
\usepackage{lipsum}
\usepackage[centerlast]{subfigure}
\usepackage[rotate,arrow,matrix]{xy}
\usepackage{tikz}

\DeclareMathOperator{\aut}{Aut}

\DeclareMathOperator{\surf}{Surf}

\newtheorem{theorem}{Theorem}[section]

\theoremstyle{definition}
\newtheorem{definition}[theorem]{Definition}

\theoremstyle{remark}
\newtheorem{remark}[theorem]{Remark}

\numberwithin{equation}{section}

\def\ol#1{\overline{#1}}
\def\ul#1{\underline{#1}}

\def\Z{\mathbb Z}
\def\Q{\mathbb Q}
\def\R{\mathbb R}

\def\S{\mathbb S}

\def\FS{\mathfrak S}

\def\M{\mathcal M}

\definecolor{light-gray}{gray}{0.80}
\definecolor{ovdarkgreen}{RGB}{106, 119, 42}

\usepackage[symbol]{footmisc}

\newcommand\blfootnote[1]{%
  \begingroup
  \renewcommand\thefootnote{}\footnote{#1}%
  \addtocounter{footnote}{-1}%
  \endgroup
}

\begin{document}

\title[Combinatorial moduli of bordered surfaces]{A combinatorial model for the moduli of bordered Riemann surfaces and a compactification}

\author[R. Kaufmann]{Ralph Kaufmann}
\address{Purdue University, Department of Mathematics, 150 N. University Street
West Lafayette, IN 47907-2067 - USA}
\email{rkaufman@math.purdue.edu}

\author[J. Z\'u\~niga]{Javier Z\'u\~niga$^*$}
\address{Departmento de Econom\'{i}a\\Universidad del Pac\'{i}fico\\
Av. Salaverry 2020, Jes\'{u}s Mar\'{i}a, Lima 11 - Per\'{u}}
\email{zuniga\underline{ }jj@up.edu.pe}

\subjclass[2020]{Primary 32G15, 57Q15; Secondary 57R18}
\date{\today}
\keywords{Bordered Riemann surfaces, Ribbon graphs, Moduli spaces}

\begin{abstract}
We construct a combinatorial moduli space closely related to the KSV-compactification of the moduli space of bordered marked Riemann surfaces. The open part arises from symmetric metric ribbon graphs. The compactification is obtained by considering sequences of non contractible subgraphs. This leads to a partial real blow-up of rational cells that together form a compact orbi-cell space. For genus zero the constructed space gives an orbi-cell decomposition of the corresponding analytic moduli space decorated by real numbers and a compactification of this space. In higher genus the relation is more involved, as we briefly explain. The spaces we construct are of interest in their own right as they are constructed directly from an interesting class of graphs.
\end{abstract}

\maketitle

%\tableofcontents

\section{Introduction}

Let\blfootnote{$^*$Corresponding author.} $\M^{dec}_{g,n}$ denote the (decorated) moduli space of Riemann surfaces of genus $g$ and $n$ marked points together with decorations, i.e., a choice of a positive real number for each marked point so that they add up to one. Let $\M^{comb}_{g,n}$ denote the (combinatorial) moduli space of metric ribbon graphs (or fat graphs) of genus $g$ with $n$ boundary cycles so that the sum of lengths of all edges adds up to one half. The combinatorial moduli space has a natural orbicell decomposition. A known result of Harer, Mumford, Thurston \cite{Harer}, Penner \cite{pen87}, Bowdich and Epstein \cite{boweps}, states that the decorated and combinatorial moduli spaces are homeomorphic. This homeomorphism can be realized by quadratic differentials with quadratic residues (the ribbon graph model) or by hyperbolic metrics with horocycles (the fat graph model). As a consequence, the space $\M^{dec}_{g,n}$ has an orbicell decomposition given by $\M^{comb}_{g,n}$. 

Konsevich's proof of Witten's conjecture in \cite{kon:itma} requires a  compactification of the combinatorial moduli space emulating the Deligne-Mumford compactification of the moduli space of Riemann surfaces. This compactification losses a lot of information (it is more coarse) when the surface degenerates in certain ways, which is codified topologically by the so called \emph{genus defect}. The details of this compactification together with a finer compactification recovering the missing data was given by Looijenga in \cite{loo}. A further extension of Looijenga's compactification was given in \cite{zun2015} using the real compactification proposed by \cite{KSV1995}. 

For bordered Riemann surfaces a few combinatorial models have been proposed including \cite{godin2007} and \cite{cos:dpv} using BW graphs, see \cite{egas2015} for a summary. These models are homotopical in nature. In this article, we propose a combinatorial model based on ribbon graphs and the double construction. The advantage is that our model identifies the homeomorphism type of a concrete space, but it can be related directly to the moduli of bordered surfaces only for genus zero with one boundary component and some special cases in genus one. We also discuss some relations between out model and the model in \cite{cos:dpv}, and also a connection with the KSV compactification of the moduli space.
This type of space is also of interest in terms of open/closed field theory and actions \cite{KP,ochoch}.

This article is structured as follows. In section 2, we review the construction of polytopes and their truncation. In section 3, we introduce ribbon graphs and in section 4, we define their moduli spaces. In section 5, we define the combinatorial blow-up of the moduli spaces in the previous section. In section 6, we review the theory of Strebel-Jenkins differentials on Riemann surfaces. In section 7, we define several moduli spaces of interest: decorated moduli spaces, moduli spaces of symmetric surfaces and different compactifications. We also define the space of interest in this article: the moduli space of symmetric ribbon graphs. In section 8, we give a few examples of these new spaces. In section 9, we establish some connections with known moduli spaces and BW graphs. 

\section{Polytopes}

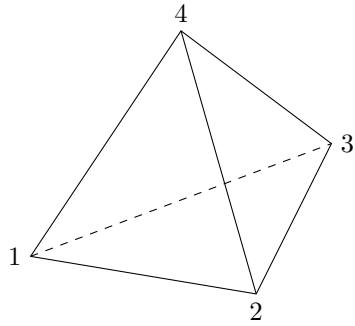
\begin{figure}
\centering
\begin{tikzpicture}
\draw (0,0) node[left]{$1$} -- (2,3) node[above]{$4$} -- (4,1.5) node[right]{$3$} -- (3,-0.5) node[below]{$2$} -- (0,0);
\draw (2,3) -- (3,-0.5);
\draw[dashed] (0,0) -- (4,1.5);
\end{tikzpicture}
\caption{The simplex $\Delta_S$ where $S = \{ 1,2,3,4 \}$.}
\label{tetrah}
\end{figure}

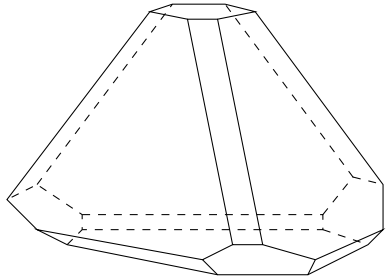
\begin{figure}
\centering
\begin{tikzpicture}[scale=0.2]

\draw (15,2) -- (17,2) -- (20.5,1) -- (20,0) -- (14,0) -- (13,1) -- (15,2) -- (12,17) -- (14,17) -- (17,2);
\draw (12,17) -- (9.5,17.5) -- (11,18) -- (14.5,18) -- (16.5,17.5) -- (14,17);
\draw (14,0) -- (4,2) -- (2,3) -- (13,1);
\draw (20,0) -- (24,2) -- (25,3) -- (20.5,1);
\draw (2,3) -- (0,5) -- (9.5,17.5);
\draw (25,3) -- (25,6) -- (16.5,17.5);

\draw[dashed] (4,2) -- (5,3) -- (5,4) -- (2,6) -- (0,5);
\draw[dashed] (2,6) -- (11,18);
\draw[dashed] (5,4) -- (21,4) -- (21,3) -- (5,3);
\draw[dashed] (21,4) -- (23,6.5) -- (14.5,18);
\draw[dashed] (21,3) -- (24,2);
\draw[dashed] (23,6.5) -- (25,6);

\end{tikzpicture}
\caption{The permutahedron $P_S$ where $S = \{ 1,2,3,4 \}$.}
\label{permutafour}
\end{figure}

Given a finite set $S$ we denote by $\Delta_S$ the standard simplex on $S$. Figure~\ref{tetrah} shows the tetrahedron. In this case if $\Delta_S$ can be given coordinates so that $x_1+x_2+x_3+x_4=1$ where $x_1 \ge 0$. Notice that as $x_1, x_3 \to 0$, we obtain the interval $\Delta_{3,4}$. For a set $S$ we denote by $P_S$ the permutohedron on $S$ which can be defined as a \emph{truncation} or \emph{blow-up} of $\Delta_S$ along all faces in increasing dimension. When $S=\{ 1,2,3,4 \}$, Figure~\ref{permutafour} shows $P_S$. The faces are obtained by considering proper, non-empty sequences of subsets of $S$. For example, the sequence $\{ 1,2,3,4 \} \supset \{ 1,3 \}$ gives the rectangle $P_{\{2,4 \}} \times P_{\{1,3 \}}$. In the previous example, as $x_1,x_3 \to 0$, the rectangle is obtained by blowing-up $\{ 1,3 \}^c = \{ 2,4 \}$.

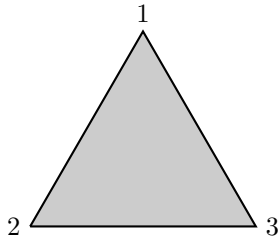
\begin{figure}
\centering
\begin{tikzpicture}[scale=1.5]
\draw[thick, fill=light-gray] (-1,0) node[left] {2} -- (1,0) node[right] {3} -- (0,1.73) node[above] {1} -- (-1,0);
\end{tikzpicture}
\caption{The simplex $\Delta_S$ where $S = \{ 1,2,3 \}$.}
\label{triangle}
\end{figure}

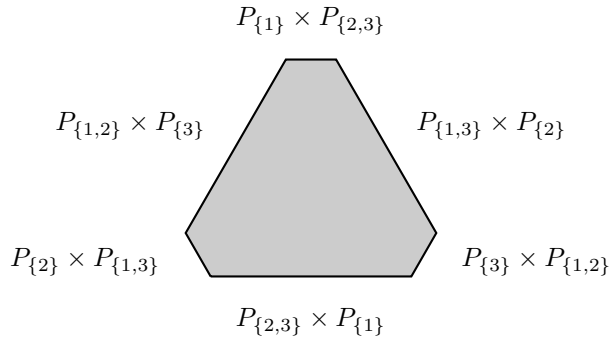
\begin{figure}
\centering
\begin{tikzpicture}[scale=2]
\draw[thick, fill=light-gray] (-0.667,0) -- (-0.833,0.289) -- (-0.167,1.443) -- (0.167,1.443) -- (0.833,0.289) -- (0.667,0) --  (-0.667,0);
\node at (0,1.7) {$P_{\{1\}} \times P_{\{2,3\}}$};
\node at (-1.5,0.1) {$P_{\{2\}} \times P_{\{1,3\}}$};
\node at (1.5,0.1) {$P_{\{3\}} \times P_{\{1,2\}}$};
\node at (0,-0.3) {$P_{\{2,3\}} \times P_{\{1\}}$};
\node at (-1.2,1) {$P_{\{1,2\}} \times P_{\{3\}}$};
\node at (1.2,1) {$P_{\{1,3\}} \times P_{\{2\}}$};
\end{tikzpicture}
\caption{The permutahedron $P_S$ where $S = \{ 1,2,3 \}$.}
\label{permutathree}
\end{figure}

From the simplex in Figure~\ref{triangle}, and using that labeling, we can obtain the permutohedron $P_S$ in Figure~\ref{permutathree} together with all faces.

\begin{figure}
\centering
\begin{tikzpicture}[scale=0.2]

\draw (12,16) -- (9.5,16) -- (8.5,17.5) -- (11,19) -- (14,18.5) -- (13.5,17) -- (12,16) -- (18,0) -- (19.5,1) -- (13.5,17);
\draw (8.5,17.5) -- (0,3.5) -- (1,2) -- (9.5,16);
\draw (1,2) -- (18,0);
\draw (19.5,1) -- (25,8.5) -- (14,18.5); 

\draw[dashed] (0,3.5) -- (22,9) -- (25,8.5);
\draw[dashed] (22,9) -- (11,19);

\end{tikzpicture}
\caption{An intermediate polytope.}
\label{qbuss}
\end{figure}
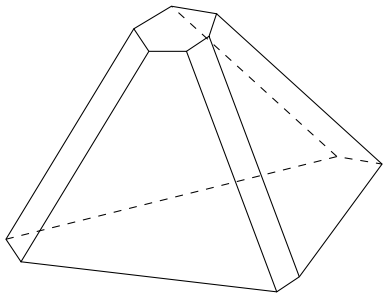

It is possible to obtain other polytopes (nestohedra) by blowing-up (truncating) only certain lower dimensional faces of a simplex. One way to indicate which faces to blow-up in $\Delta_S$ is by providing a subset $B \subset 2^S \setminus \{\varnothing\}$. 

\begin{definition}
For a given set $S$, let $N_B$ be the result of blowing-up $\Delta_S$ along the faces determined by $B$ and increasing by dimension.  
\end{definition}

With the definition above, $N_\varnothing = \Delta_S$ and $N_{2^S \setminus \{ \varnothing \}} = P_S$. For an intermediate example, if we consider $S=\{1,2,3,4 \}$ and \[ B = \{ \{4\}, \{1,4\}, \{2,4\}, \{3,4\} \} \] we obtain the polytope in Figure~\ref{qbuss} by blowing up first the top vertex of the tetrahedron in Figure~\ref{tetrah} and then the corresponding three edges.

\begin{remark}
It is possible to use the language of hypegraph polytopes as in \cite{coi2019} to deal with the polytopes in this sections.
\end{remark}

\section{Ribbon graphs}

\begin{definition}
A \emph{ribbon graph} $\Gamma$ is a connected graph with vertices of valence at least three together with a subdivision of edges into half-edges, also called \emph{edge refinement}, and a cyclic order on the set of incident half-edges to a vertex. 
\end{definition}

If $H$ is the set of half-edges and $v$ is a vertex of $\Gamma$, let $H_v$ be the set of half-edges adjacent to this vertex. A cyclic ordering at a vertex $v$ is an ordering of $H_v$ up to cyclic permutation. Once a cyclic ordering of $H_v$ is chosen, a cyclic permutation of $H_v$ is defined (an element of $\FS_{H_v}$): it moves a half-edge to the next in the cyclic order. Define by $\sigma_0$ the element of $\FS_H$ (the group of permutations acting on $H$)  which is the product of all the cyclic permutations at every vertex. Also, let $\sigma_1$ be the involution that interchanges the two half-edges on each edge of $\Gamma$. This combinatorial data completely defines the ribbon graph. 

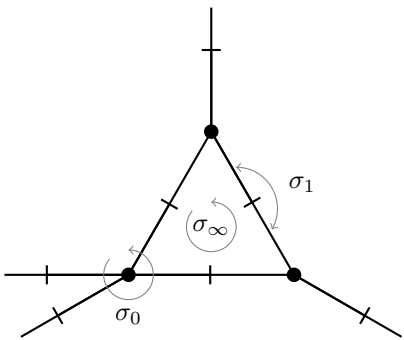
\begin{figure} 
	\begin{center}
		\begin{tikzpicture}[scale=2.2]
			\draw[thick,-|] (1,0) -- (1+.65/3*2,-0.75/3);
			\draw[thick,-|] (0,0) -- (0.5,0);
			\draw[thick,-|] (0,0) -- (-0.5,0);
			\draw[thick,-|] (0,0) -- (0.25,0.433);
			\draw[thick,-|] (0.5,0.866) -- (0.75,0.433);
			\draw[thick,-|] (0.5,0.866) -- (0.5,0.866+0.5);
			\draw[thick,-|] (0,0) -- (-.65/3*2,-0.75/3);

			\draw[thick] (1,0) -- (0,0) -- (0.5,0.866) -- (1,0);
			\draw[thick] (1,0) -- (1.65,-0.75/2);
			\draw[thick] (0.5,0.866) -- (0.5,0.866+0.75);
			\draw[thick] (0,0) -- (-0.65,-0.75/2);
			\draw[thick] (0,0) -- (-0.75,0);

			\fill (0,0) circle (1.25pt);
			\fill (1,0) circle (1.25pt);
			\fill (0.5,0.866) circle (1.25pt);
			
			\draw[<-, gray] (0,0.15) arc (90:-220:0.15);
			\node at (0,-0.25) {$\sigma_0$};
			\draw[<-, gray] (0.5,0.44) arc (90:-220:0.15);
			\node at (0.5,0.866/3) {$\sigma_\infty$};
			\draw[<->, gray] (0.65,0.65) arc (90:-30:0.25);
			\node at (1.05,0.55) {$\sigma_1$};
	\end{tikzpicture} \end{center}
	\caption[Combinatorial data]{An example of the permutations that define a ribbon graph around a boundary cycle.} \label{permutations}
\end{figure}

To be more precise, a ribbon graph  $\Gamma$ can be build out of combinatorial data in the following way. Let $H$ be a finite set of even cardinality. Let $\sigma_1 \in \FS_H$ be an involution without fixed points and $\sigma_0 \in \FS_{H}$ be such that $\sigma_0$ is a product of cyclic permutations with disjoint support. A vertex of $\Gamma$ is then given as an orbit of $\sigma_0$, while an edge is an orbit of $\sigma_1$. The set of vertices can be identified with $V(\Gamma)=H/\sigma_0$ and the set of edges with $E(\Gamma)=H/\sigma_1$. The graph thus obtained is required to be connected. 

Let $\sigma_\infty = \sigma_0^{-1} \sigma_1 \in \FS_{H}$. The set of \emph{cusps} is defined as $C(\Gamma)= H/\sigma_\infty$. The half-edges in the orbit of a cusp form a cyclically ordered set of half-edges called a \emph{boundary cycle}. Given a boundary cycle, all edges and vertices that are associated to those half-edges via $\sigma_0$ or $\sigma_1$ form a graph called a \emph{boundary subgraph}. For example, in Figure~\ref{permutations} the boundary cycle represented in the middle (and corresponding to three half edges) forms a boundary subgraph with three edges and three vertices (the middle triangle as a subgraph). Notice also that knowing $\sigma_1$ and $\sigma_\infty$ completely determines 
the ribbon graph structure since $\sigma_0= \sigma_1 \sigma_\infty^{-1}$.

An \emph{orientation of an edge} can be defined as an order on its corresponding half edges 
and we can use the notation $\vec{e} = \overrightarrow{h \sigma_1(h)}$ where $h$ is a half-edge. The involution $\sigma_1$ switches the orientation of an edge.

A \emph{loop} is an edge incident to only one vertex. A \emph{tree} is a connected graph $T$ with trivial $\bar{H}_*(T)$. 

\begin{definition}
An \emph{isomorphism} of ribbon graphs is a graph isomorphism of the edge refinement preserving the cyclic orders on each vertex. 
\end{definition}

Therefore, two graphs $\Gamma$, $\Gamma'$ are isomorphic when there is a bijection $\eta: H \to H'$ between the set of half-edges of these two graphs that commutes with $\sigma_0$, $\sigma_0'$ and $\sigma_1$, $\sigma_1'$. In particular, the boundary cycles are preserved, \emph{i.e.}, $\eta$ also commutes with $\sigma_\infty$, $\sigma_\infty'$.

\begin{figure}
	\centering
	\begin{tikzpicture}[scale=1.8]
		\draw[thick] (0,0) circle (0.5);
		\draw[fill=black] (-0.5,0) circle (1.2pt);
		\draw[fill=black] (0.5,0) circle (1.2pt);
		\draw[thick] (-0.5,0) -- (0.5,0);
		
		\draw[->, gray] (0.55,-0.15) arc [radius=0.15, start angle=-70, end angle= 250];
		\draw[->, gray] (-0.45,-0.15) arc [radius=0.15, start angle=-70, end angle= 250];
		
		\draw[thick] (0,0.45) -- (0,0.55);
		\draw[thick] (0,-0.05) -- (0,0.05);
		\draw[thick] (0,-0.55) -- (0,-0.45);
		
		\node at (0,-1) {$\theta$ graph};
	\end{tikzpicture} \hspace{1cm}
	\begin{tikzpicture}[scale=1.8]
		\draw[thick] (0,0) circle (0.5);
		\draw[thick] (-0.5,0) to [out=15, in=215] (0.12,0.40);
		\draw[thick] (0.32,0.48) to [out=30 , in=90] (0.75,0.28) to [out=270, in=25] (0.5,0);
		\draw[fill=black] (-0.5,0) circle (1.2pt);
		\draw[fill=black] (0.5,0) circle (1.2pt);
		\node at (0,-1) {twisted $\theta$};
	\end{tikzpicture}
	\caption{Two non-isomorphic ribbon graphs of different topological types.}
	\label{rgraphs}
\end{figure}
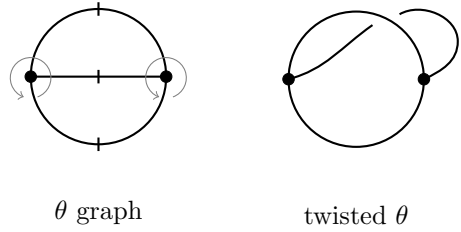

Let $\aut(\Gamma)$ denote the automorphism group of the edge refinement as a usual graph. The automorphism group of a ribbon graph, denoted by $\aut^{rg}(\Gamma)$, is the group of automorphisms of the edge refinement that also preserve the cyclic order on every vertex. Figure~\ref{rgraphs} shows two isomorphic graphs that are not isomorphic as ribbon graphs.

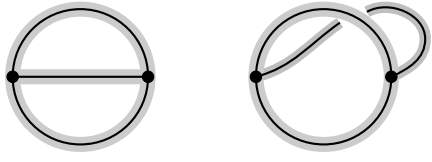
\begin{figure}
\centering
\begin{tikzpicture}[scale=1.8]

\draw[line width=0.2cm, color=light-gray] (0,0) circle (0.5);
\draw[line width=0.2cm, color=light-gray] (-0.5,0) -- (0.5,0);

\draw[thick] (0,0) circle (0.5);
\draw[fill=black] (-0.5,0) circle (1.2pt);
\draw[fill=black] (0.5,0) circle (1.2pt);
\draw[thick] (-0.5,0) -- (0.5,0);
\end{tikzpicture} \hspace{1cm}
\begin{tikzpicture}[scale=1.8]
\draw[line width=0.15cm, color=gray!50] (-0.5,0) to [out=15, in=215] (0.12,0.40);
\draw[line width=0.15cm, color=gray!50] (0.32,0.48) to [out=30 , in=90] (0.75,0.28) to [out=270, in=25] (0.5,0);
\draw[line width=0.2cm, color=light-gray] (0,0) circle (0.5);

\draw[thick] (0,0) circle (0.5);
\draw[thick] (-0.5,0) to [out=15, in=215] (0.12,0.40);
\draw[thick] (0.32,0.48) to [out=30 , in=90] (0.75,0.28) to [out=270, in=25] (0.5,0);
\draw[fill=black] (-0.5,0) circle (1.2pt);
\draw[fill=black] (0.5,0) circle (1.2pt);
\end{tikzpicture}
\caption{Two surfaces obtained from thickening ribbon graphs.}
\label{thickening}
\end{figure}

\begin{figure}
\centering
\begin{tikzpicture}[scale=1]
\draw[thick] (0,0) circle (0.5);
\draw[thick] (-0.5,0) -- (0.5,0);
\draw[fill=black] (-0.5,0) circle (2pt);
\draw[fill=black] (0.5,0) circle (2pt);

\draw[thick,xshift=2.5cm] (0,0) to [out=135, in=90] (-1,0) to [out=270,in=225] (0,0);
\draw[thick,xshift=3cm,rotate=180] (0,0) to [out=135, in=90] (-1,0) to [out=270,in=225] (0,0);
\draw[thick,xshift=2.5cm] (0,0) -- (0.5,0);
\draw[fill=black,thick,xshift=2.5cm] (0,0) circle (2pt);
\draw[fill=black,thick,xshift=2.5cm] (0.5,0) circle (2pt);

\draw[thick,xshift=6cm] (0,0) to [out=135, in=90] (-1,0) to [out=270,in=225] (0,0);
\draw[thick,xshift=6cm,rotate=180] (0,0) to [out=135, in=90] (-1,0) to [out=270,in=225] (0,0);
\draw[fill=black,thick,xshift=6cm] (0,0) circle (2pt);
\end{tikzpicture}
\caption{Non-isomorphic ribbon graphs of type $(0,3)$.}
\label{threegraphs}
\end{figure}
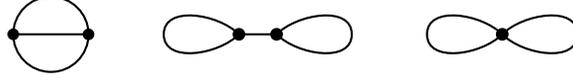

\begin{figure}
\centering
\begin{tikzpicture}[scale=1]
\draw[thick] (0,0) circle (0.5);
\draw[thick] (-0.5,0) to [out=15, in=215] (0.12,0.40);
\draw[thick] (0.32,0.48) to [out=30 , in=90] (0.75,0.28) to [out=270, in=25] (0.5,0);
\draw[fill=black] (-0.5,0) circle (2pt);
\draw[fill=black] (0.5,0) circle (2pt);
\end{tikzpicture}\qquad
\begin{tikzpicture}[scale=1]
\draw [thick] plot [smooth, tension=1] coordinates {(0,0) (0.35,0.8) (0.9,0.9) (0.8,0.35) (0,0)};
\draw [thick] plot [smooth, tension=0.5] coordinates {(0,0) (0.4,0.4) (0.75,0.45)};
\draw [thick] plot [smooth, tension=0.5] coordinates {(0,0) (0.2,-0.1) (0.4,-0.15) (0.7,-0.1) (1,0.15) (1.05,0.25) (1,0.4)};
\draw[fill=black] (0,0) circle (2pt);
\end{tikzpicture}
\caption{Non-isomorphic ribbon graphs of type $(1,1)$.}
\label{twographs}
\end{figure}
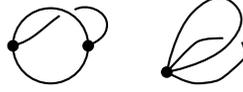

When replacing the edges by \emph{ribbons} and gluing them according to the cyclic ordering, one obtains an orientable surface of genus $g$ with $n$ holes as in Figure~\ref{thickening}. Each hole can be retracted to a boundary subgraph. Since the surface retracts to the graph, its Euler characteristic is $\chi = 2-2g-n$. Figure~\ref{threegraphs} shows three non-isomorphic graphs of type $(0,3)$: the theta graph, the double noose and the figure eight respectively. All ribbon graphs of type $(1,1)$ are shown in Figure~\ref{twographs}: The twisted theta graph and the twisted figure eight. 

Analogously to a labeling of the marked points on a surface, the boundary cycles or cusps of a ribbon graph can be \emph{labeled}. 

\begin{definition}
For a labeled ribbon graph, let $\aut^{rg}_\partial(\Gamma)$ denote the ribbon graph automorphism group preserving the labels. 
\end{definition}

It can be shown that up to isomorphism, there is exactly one labeled theta graph and three labeled double nooses. Since the twisted theta graph has only one boundary cycle, there is only one twisted labeled theta graph up to isomorphism. From now on, we assume that all ribbon graphs have labeled boundary cycles. 

A metric on a ribbon graph is a function $l : E(\Gamma) \to \R_+$. A metric can be represented by a point in $\R_+^{|E(\Gamma)|}$. A metric is \emph{unital} when the sum of the lengths of the boundary cycles is equal to one. This is equivalent to requiring that the sum of all edges equals one half since each edge is counted twice in the sum of the lengths of boundary cycles. A unital metric can be then represented by a point in $int ( \Delta_{E(\Gamma)})$.

\begin{figure}
	\centering
	\includegraphics{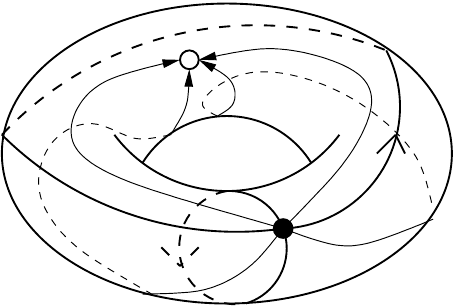}
	\caption{Surface obtained from a metric twisted figure eight.}
	\label{torussurface}
\end{figure}

To every metric ribbon graph $\Gamma$ we associate an oriented punctured surface $\surf(\Gamma)$ as follows. To each oriented edge $\vec{e} = \overrightarrow{h \sigma_1(h)}$ we attach an oriented triangle $K_{\vec{e}}=C |\vec{e}|$, 
where $|\vec{e}|$ is a segment of length $l(e)$ and $C$ refers to the cone construction. 
The base of $K_{\vec{e}}$ is identified with the base of $K_{\sigma_1(\vec{e})}$. 
Next we paste the right-hand edge of $K_{\vec{e}}$ with the left-hand edge of $K_{\sigma_\infty(\vec{e})}$. All vertices opposite to the bases of all triangles in the same boundary cycle are identified. Such points are in one to one correspondence with the orbits of $\sigma_\infty$; which is why we call them cusps. Gluing triangles is done in a way compatible with the orientation of each triangle, and thus, this punctured surface is triangularized. Let $v'=|V(\Gamma)|, e'=|E(\Gamma)|$ and $n'$ denote the number of vertices, edges and cusps, respectively. It is easy to shows that there are $2e'$ faces, $3e'$ edges and $v'+n'$ points. The surface $\surf(\Gamma)$ has genus $g=(2-v'+e'-n')/2$. This surface comes with a natural orientation given by the tiles since they are naturally oriented and their orientations match each other because of the way we have glued them. Figure~\ref{torussurface} shows the surface obtained from a metric twisted figure eight. The meridian and longitude joint at the dark point in the front of the torus and form the metric twisted figure eight. There are four cones based at the ribbon graph, their walls run from the dark point in the front of the torus to the hollow point at the back.

\section{Combinatorial moduli space}

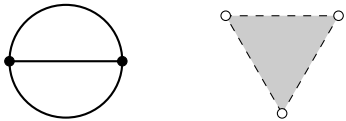
\begin{figure}
\centering
\begin{tikzpicture}[scale=1.5]
\draw[thick] (0,0) circle (0.5);
\draw[fill=black] (-0.5,0) circle (1.2pt);
\draw[fill=black] (0.5,0) circle (1.2pt);
\draw[thick] (-0.5,0) -- (0.5,0);
\end{tikzpicture} \hspace{1cm}
\begin{tikzpicture}[scale=1.5]

\draw[dashed, fill=light-gray] (0,0) -- (-0.5,0.866) -- (0.5, 0.866) -- (0,0);

\draw[fill=white] (0,0) circle (1.2pt);
\draw[fill=white] (-0.5,0.866) circle (1.2pt);
\draw[fill=white] (0.5,0.866) circle (1.2pt);

\end{tikzpicture}
\caption{Rational cell of the theta graph.}
\label{rationalthetacell}
\end{figure}

\begin{definition}
For a given ribbon graph $\Gamma$, its associated \emph{rational cell} is the quotient space $int \, (\Delta_{E(\Gamma)}) / \aut^{rg}_\partial(\Gamma)$ where $int \, (\Delta_{E(\Gamma)})$ is the space of unital metrics of $\Gamma$ and the action is by given by permutation of edges.
\end{definition}

 Figure~\ref{rationalthetacell} shows the theta graph and its corresponding rational cell. The automorphism group is trivial in this case. 

The contraction of an edge that is not a cycle results in a new ribbon graph (with induced cyclic orders on each vertex) of the same topological type. This process is called \emph{edge collapse}. The process can be generalized to the collapse of any disjoint union of subtrees. The edge collapse can be used to glue rational cells. For example, the collapse of one edge of the theta graph or the middle edge in the double noose results in the figure eight graph. This induces a gluing of a higher dimensional rational cell to a lower dimensional one. The orbispace of metric ribbon graphs of a fixed topological type $(g,n)$ up to isomorphism is the combinatorial moduli space $\M^{comb}_{g,n}$ obtained by gluing all rational cells of type $(g,n)$. The orbicell structure is studied in detail in \cite{mupe}.

\begin{figure}
\centering
\begin{tikzpicture}[scale=2.2]

\draw[thick, dashed, fill=light-gray] (-1,0) -- (0,1.732) -- (1,0) -- (-1,0);
\draw[thick] (0,0) -- (-0.5,0.866) -- (0.5,0.866) -- (0,0);

\draw[fill=white] (0,0) circle (1.2pt);
\draw[fill=white] (-1,0) circle (1.2pt);
\draw[fill=white] (1,0) circle (1.2pt);
\draw[fill=white] (-0.5,0.866) circle (1.2pt);
\draw[fill=white] (0.5,0.866) circle (1.2pt);
\draw[fill=white] (0,1.732) circle (1.2pt);

\end{tikzpicture}
\caption{The orbispace $\M^{comb}_{0,3}$.}
\label{combmodulitheta}
\end{figure}
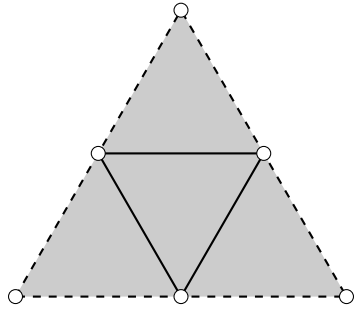

Figure~\ref{combmodulitheta} shows the combinatorial moduli space $\M^{comb}_{0,3}$. The two-dimensional rational cell in the middle corresponds with the theta graph. The other three two-dimensional rational cells correspond with the three non-isomorphic ways to label the boundary cycles of the double noose graph. The three one-dimensional rational cells correspond with the three non-isomorphic ways to label the boundary cycles of the figure eight graph.

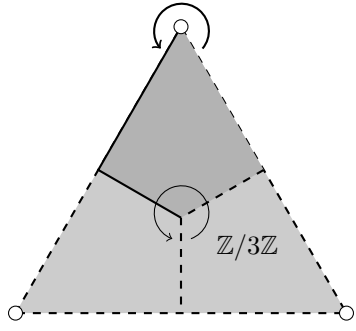
\begin{figure}
\centering
\begin{tikzpicture}[scale=2.2]

\draw[thick, dashed, fill=light-gray] (-1,0) -- (0,1.732) -- (1,0) -- (-1,0);

\fill[color=gray!60] (-0.5,0.866) -- (0,1.732) -- (0.5,0.866) -- (0,0.577);

\draw[thick] (0,1.732) -- (-0.5,0.866) -- (0,0.577);
\draw[thick, dashed] (0.5,0.866) -- (0,0.577) -- (0,0);

\draw[fill=white] (-1,0) circle (1.2pt);
\draw[fill=white] (1,0) circle (1.2pt);

\draw[fill=white] (0,1.732) circle (1.2pt);

\draw[->] (0.06,0.45) arc [radius=0.165, start angle=-70, end angle= 250];
\node at (0.4,0.4) {$\Z/ 3\Z$};
\draw[thick, ->] (0.13,1.6) arc [radius=0.165, start angle=-40, end angle= 220];

\end{tikzpicture}
\caption{The orbispace $\M^{comb}_{1,1}$.}
\label{ellipticommoduli}
\end{figure}

Figure~\ref{ellipticommoduli} shows the combinatorial moduli space $\M^{comb}_{1,1}$. There is only one two-dimensional rational cell corresponding with the twisted theta graph and depicted as the whole triangle. The automorphism group of the twisted theta graph is $\Z/ 3\Z$. The action of the automorphism group on the rational cell is by rotations, thus the shaded region represents the fundamental domain. This actions glues the lower left edge of the fundamental domain to the lower right edge. The top half of the left edge of the fundamental domain corresponds with the twisted figure eight. The automorphism group of the twisted figure eight is $\Z / 2\Z$ and its action glues the top half of the left edge to top half of the right edge. As a result, $\M^{comb}_{1,1}$ is homeomorphic to a sphere minus a point.

\section{Combinatorial blow-up}

Given a ribbon graph $\Gamma$, define $A \subset 2^{E(\Gamma)} \setminus \{ \varnothing \}$ as the family of subsets of edges corresponding with subgraphs of $\Gamma$ that are not a union of trees. Let $B$ be the family defined by taking the complement of every element in $A$. 

\begin{definition}
Since the action of $\aut^{rg}_\partial(\Gamma)$ can be extended to the boundary of $N_B$, we define the \emph{compact rational cell} associated to $\Gamma$ as $N_B / \aut^{rg}_\partial(\Gamma)$. Gluing compact rational cells along the edge collapse of the union of trees, we obtain the compact combinatorial moduli space \[ \underline{\M}^{comb}_{g,n} = \coprod_{\Gamma \,\, \text{of type} \,\, (g,n)} N_B/\aut^{rg}_\partial(\Gamma). \]
\end{definition}

The subspace $\underline{\M}^{comb}_{g,n} \setminus \M^{comb}_{g,n}$ is formed by new cells created by the combinatorial blow-up. It is possible to identified these new cells by combinatorial objects that are generalizations of metric ribbon graphs. These new graphs are in fact more general than the semistable ribbon graphs defined in \cite{zun2015}, but we will not try to define them here in full generality. This will be addressed in future work, here we will be content with a description in terms of truncations given by chains of inclusions.

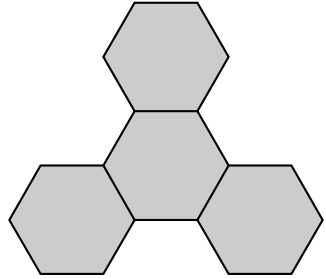
\begin{figure}
\centering
\begin{tikzpicture}[scale=2.5]

\draw[fill=light-gray, thick] (0.167,0.287) -- (0.333,0) -- (0.667,0) -- (0.833,0.287) -- (0.667,0.577) -- (0.333,0.577) -- (0.167,0.866) -- (0.333,1.155) -- (0.167,1.443) -- (-0.167,1.443) -- (-0.333,1.155) -- (-0.167,0.866) -- (-0.333,0.577) -- (-0.667,0.577) -- (-0.833,0.287) -- (-0.667,0)-- (-0.333,0) -- (-0.167,0.287) -- (0.167,0.287);

\draw[thick] (0.167,0.287) -- (0.333,0.577);
\draw[thick] (0.167,0.866) -- (-0.167,0.866);
\draw[thick] (-0.167,0.287) -- (-0.333,0.577);

\end{tikzpicture}
\caption{The orbispace $\underline{\M}^{comb}_{0,3}$.}
\label{compactmodulitheta}
\end{figure}

Figure~\ref{compactmodulitheta} shows $\underline{\M}^{comb}_{0,3}$. In this case all vertices are truncated because the only subgraphs that are unions of trees are single edges. In this case all automorphism groups are trivial. 

\begin{figure}
\centering
\begin{tikzpicture}[scale=2]

\draw[thick, dashed, fill=light-gray] (-0.667,0) -- (-0.833,0.289) -- (-0.167,1.443) -- (0.167,1.443) -- (0.833,0.289) -- (0.667,0) --  (-0.667,0);

\fill[color=gray!60] (-0.5,0.866) -- (-0.167,1.443) -- (0.167,1.443) -- (0.5,0.866) -- (0,0.577);
\draw[thick] (0.167,1.443) -- (-0.167,1.443) -- (-0.5,0.866) -- (0,0.577);
\draw[thick, dashed] (0.5,0.866) -- (0,0.577) -- (0,0);

\draw[->] (0.06,0.45) arc [radius=0.165, start angle=-70, end angle= 250];
\node at (0.4,0.4) {$\Z/ 3\Z$};

\draw[thick, ->] (0.33,1.3) arc [radius=0.44, start angle=-40, end angle= 220];

\end{tikzpicture}
\caption{The orbispace $\underline{\M}^{comb}_{1,1}$.}
\label{ellipticcompactmoduli}
\end{figure}
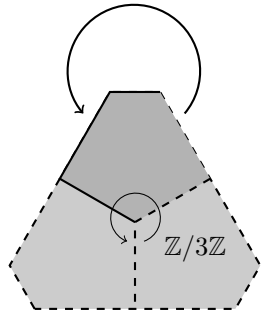

The space $\underline{\M}^{comb}_{1,1}$ can be seen Figure~\ref{ellipticcompactmoduli}. The vertex of the single two-dimensional compact rational cell is truncated. The gluings due to the group actions of $\Z / 3\Z$ and $\Z/ 2 \Z$ work as before, but in this case a boundary isomorphic to a circle is created due to the truncation. The resulting compact moduli space is homeomorphic to a surface minus an open disc.

\section{Quadratic differentials}

\begin{definition}
A \emph{meromorphic quadratic differential} on a Riemann surface $\Sigma$ is a meromorphic section of $(T^* \Sigma)^{\odot 2}$, the second symmetric power of the cotangent bundle.
\end{definition}

 The notions of zero and order of a zero do not depend on the local representation of a quadratic differential. In the same way, the notion of pole and order of a pole are invariant by change of coordinates. Zeros and poles are called \emph{critical points}. A pole of order two of a quadratic differential is called a \emph{double pole} and a pole of order one is called a \emph{simple pole}. Given a representation in local coordinates $f(z)dz^2$ around a double pole $q$, we can express $f$ as \[ f(z) = \frac{a_{-2}}{z^2} + \frac{a_{-1}}{z} + a_0 + \cdots \] and call the term $a_{-2}$ its \emph{quadratic residue}. It can be shown that this number does not depend on the choice of local coordinates. 

Quadratic differentials define certain curves on the Riemann surface. If $q=f(z)dz^2$ is a meromorphic quadratic differential, then the parametric curve $\vec{r}:(a,b) \to C$ is called a \emph{horizontal trajectory} or \emph{leaf} of $q$ if \[ f(\vec{r}(t)) \left( \frac{d\vec{r}(t)}{dt} \right)^2 > 0 \] and a \emph{vertical trajectory} if \[ f(\vec{r}(t)) \left( \frac{d\vec{r}(t)}{dt} \right)^2 < 0. \] 

\begin{definition}
A \emph{Strebel-Jenkins differential} is a meromorphic quadratic differential with only double poles with negative quadratic residues. 
\end{definition}

Strebel-Jenkins differentials have only two kinds of leaves: closed ones (surrounding a double pole) and critical ones (connecting zeroes). The union of critical leaves and zeroes forms the \emph{critical graph}. The vertical trajectories connect the double poles to the critical graph and are orthogonal to the closed leaves under the metric induced by $\sqrt{q}$. The following existence and uniqueness theorem follows from the work of Jenkins and Strebel (see \cite{strebel} and \cite[Theorem 7.6]{loo}).

\begin{figure}
	$$
	\begin{array}{ccc}
		\begin{tikzpicture}[scale=1.75]

\begin{scope}

\clip (0,0) circle (1cm);

\foreach \i in {-3,...,3}
{\draw (-1, \i/4) -- (1, \i/4); }

\foreach \i in {-3,...,3}
{\draw[dashed] (\i/4, -1) -- (\i/4, 1); }

\fill (0,0) circle (1.5pt);

\end{scope}

\draw[thick] (0,0) circle (1cm);

\end{tikzpicture} & \quad \begin{tikzpicture}[scale=1.75]

\draw[thick] (0,0) circle (1cm);
\fill (0,0) circle (1.5pt);

\draw (0,0) -- (1,0);
\draw (0,0) -- (-0.5,0.866);
\draw (0,0) -- (-0.5,-0.866);

\draw[dashed] (0,0) -- (-1,0);
\draw[dashed] (0,0) -- (0.5,0.866);
\draw[dashed] (0,0) -- (0.5,-0.866);

\end{tikzpicture} & \quad \begin{tikzpicture}[scale=1.75]

\draw[thick] (0,0) circle (1cm);
\fill (0,0) circle (1.5pt);

\foreach \i in {1,...,4}
{\draw (0,0) circle (\i/5); }

\draw[dashed] (0,0) -- (1,0);
\draw[dashed] (0,0) -- (-1,0);
\draw[dashed] (0,0) -- (0,1);
\draw[dashed] (0,0) -- (0,-1);
\draw[dashed] (0,0) -- (0.7071,0.7071);
\draw[dashed] (0,0) -- (-0.7071,-0.7071);
\draw[dashed] (0,0) -- (-0.7071,0.7071);
\draw[dashed] (0,0) -- (0.7071,-0.7071);

\end{tikzpicture} \\
		q = dz^2 & \quad q = z dz^2 & \quad q = - \frac{dz^2}{z^2}
	\end{array}
	$$
	\caption{Different behaviors of Strebel-Jenkins differentials. The solid lines represent horizontal trajectories and the dotted one vertical trajectories.}
\end{figure}
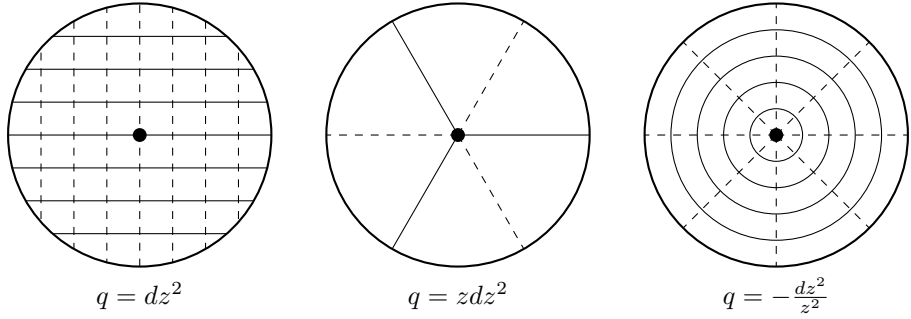

\begin{theorem}[Strebel] \label{strebel}
	Given a Riemann Surface of genus $g$ with $n$ marked points and decorations in $int(\Delta^{n-1})$, there exists a unique quadratic differential with the following properties. It is holomorphic in the puncture surface. The union of closed leaves form semi-infinite cylinders around the marked points. The quadratic residues coincide with the decorations.
\end{theorem}

\section{Moduli spaces and decorations}

Let $\M_{g,n}$ denote the moduli space of Riemann surfaces of genus $g$ and $n$ marked points. The topological type of these surfaces is defined as $(g,n)$. The interior of a Riemann surface of type $(g,n)$ is the result of taking away the marked points, thus obtaining a \emph{punctured} surface. 

Following \cite{liu}, let $\R \M_{(g,b,k), (n,m)}$ denote the moduli of smooth symmetric compact Riemann surfaces with $2n + m$ marked points where a symmetric surface is a pair $(\Sigma, \sigma)$ of a surface $\Sigma$
and $\sigma$ is an antiholomorphic involution.
The index $g$ fixes the genus of $\Sigma$ and $b$ is the number of components of the fixed points of $\sigma$ that is, $\Sigma^\sigma \cong \sqcup_b \S^1$.
The index  $k$ fixes the orientability,  it is $0$ if $\Sigma/ \langle \sigma \rangle$ is orientable and it is $1$ if $\Sigma/ \langle \sigma \rangle$ is not orientable. The symmetry interchanges $n$ pairs of marked points away from $\Sigma^\sigma$ and fixes $m$ marked points in $\Sigma^\sigma$. 

Denote by $\R \M_{g,n,m}$ the moduli space of symmetric surfaces of genus $g$ with $(n,m)$ marked points. This space is a disjoint union of the spaces $\R \M_{(g,b,k), (n,m)}$ ranging over all the different topological types depending on $b$ and $k$. Forgetting the symmetry induces a map $f : \R \M_{g,n,m} \to \M_{g,2n+m}$ that is generically injective, but it fails to be injective when the automorphism group of the marked surface $\Sigma$ is bigger than the automorphism group of the marked symmetric surface $(\Sigma, \sigma)$. 

Following \cite{liu} again, let $\M_{g,b,n,\vec{m}}$ be the moduli space of bordered Riemann surfaces of genus $g$ with $b$ boundary components isomorphic to $\S^1$, $n$ marked points away from the boundary and $\vec{m}$ marked points on the boundary. Here the boundary components are also marked. 

There is an alternative theory in terms of hyperbolic geometry which leads to the same combinatorics, cf. \cite{pennerbordered, KP}.

\begin{figure}
	\centering
\begin{tikzpicture}
	
	\draw [thick] plot [smooth, tension=0.75] coordinates {(2.5,0.6) (1,1) (0,0) (1,-1) (2.5,-0.6)};
	\draw [thick] (1.6,0) arc [radius=0.5, start angle=320, end angle= 220];
	\draw [thick] (1.5,-0.1) arc [radius=0.35, start angle=40, end angle= 140];
	
	\draw [thick] (2.5,-0.6) arc [radius=0.94, start angle=220, end angle=140];
	\draw [thick] (2.5,-0.6) arc [radius=0.94, start angle=320, end angle=400];

	\draw [thick] plot [smooth, tension=0.75] coordinates {(3.5,0.6) (5,1) (6,0) (5,-1) (3.5,-0.6)};
	\draw [thick] (5.1,0) arc [radius=0.5, start angle=320, end angle= 220];
	\draw [thick] (5,-0.1) arc [radius=0.35, start angle=40, end angle= 140];
	
	\draw [thick] (3.5,-0.6) arc [radius=0.94, start angle=220, end angle=140];
	\draw [thick, dashed] (3.5,-0.6) arc [radius=0.94, start angle=320, end angle=400];
	
	\draw[fill=black] (0,0) circle (1.5pt);
	\draw[fill=black] (6,0) circle (1.5pt);
	\draw[fill=black] (2.5,-0.6) circle (1.5pt);
	\draw[fill=black] (3.5,-0.6) circle (1.5pt);
	
\end{tikzpicture}
\caption{The double of a surface of type $(1,1,1,1)$.}
\label{complexdouble}
\end{figure}
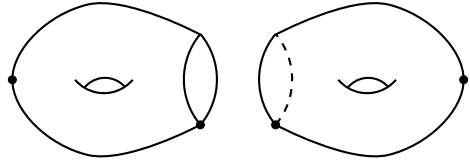

The \emph{complex double} or simply \emph{double} of a surface of topological type $(g,b,n,\vec{m})$ is a surface of topological type $(2g+b-1, 2n+m)$. The doubled surface has a complex structure and a real involution $\sigma$ whose fixed points lie in the boundary of the bordered surface as in Figure~\ref{complexdouble}. By definition, the Euler characteristic of the bordered surface is \[ \chi = 2-2g-b-n-m/2,\] this is half the Euler characteristic of the doubled surface. A doubled surface is in fact symmetric. This observations induces a finite to one map $d: \M_{g,b,n,\vec{m}} \to \R \M_{2g+b-1, n, m}$. 

By considering stable Riemann surfaces with possible nodal singularities one can define the Deligne-Mumford compactifications of these spaces (\cite{abikoff1980}, \cite{seppala1991}, \cite{silhol1992}, \cite{liu}). Moreover, the composition of the double and forgetful map  \[ \ol{\M}_{g,b,n,\vec{m}} \stackrel{d}{\longrightarrow} \ol{\R \M}_{(2g+b-1, b, 0), (n, m)} \stackrel{f}{\longrightarrow} \ol{\M}_{2g+b-1, 2n +m}\] has similar characteristics as before, in particular the double map is generically two to one. 

The KSV compactification $\ul{\M}_{g,n}$ is obtained by decorating the nodes of a stable surface with relative phase parameters \cite{KSV1995}. Geometrically, this is the same as a real tangent direction: a ray in the complex tensor product of the tangent spaces on each side of a node. 

The decorated moduli space of Riemann surfaces of genus $g$ and $n$ marked points is by definition \[ \M_{g,n}^{dec} = \M_{g,n} \times int (\Delta^{n-1}).\] Each marked point is decorated by a positive real number so that the sum of decorations is one. In section 5 of \cite{mupe}, a homeomorphism $\M_{g,n} \times \R^n_+ \to RGB^{met}_{g,n}$ is described. By restricting the domain to the standard simplex, we obtain the map $\varphi : \M^{dec}_{g,n} \to \M^{comb}_{g,n}$. 

This map is realized by Strebel's theory of quadratic differentials. Given a Riemann surface $\Sigma$ with decorations, Strebel's theorem produces a Strebel-Jenkins differential whose critical graph is a metric ribbon graph $\Gamma$. For the inverse, given a metric ribbon graph $\Gamma$, it is possible to give complex charts to $\surf(\Gamma)$ to obtain a Riemann surface $\Sigma$. The decorations of each marked point is the length of the corresponding boundary subgraph. 

\begin{figure}
	\begin{center}
		\includegraphics{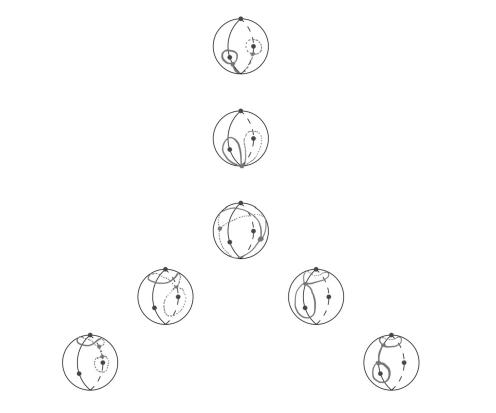}
	\end{center}
	\caption{Critical graphs on a surface of type $(0,3)$ corresponding with the seven rational cells of $\M^{comb}_{0,3}$. The critical graphs appear in a lighter shade of grey on the surface of Riemann spheres with three marked points.}
	\label{degenerationsphere}
\end{figure}

Each rational cell of the moduli space $\M^{comb}_{0,3}$ depicted in Figure~\ref{combmodulitheta} can be represented by the critical graph of a Strebel-Jenkins differential on the Riemann surface with three marked points. This is shown in Figure~\ref{degenerationsphere} and illustrates how the map $\varphi$ works. This last figure also illustrates how the metric ribbon graphs deform as we move around the moduli space. 

\begin{figure}
	\centering
	\begin{tikzpicture}[scale=2.8]
		
		\draw[fill=light-gray, thick] (0.167,0.287) -- (0.333,0) -- (0.667,0) -- (0.833,0.287) -- (0.667,0.577) -- (0.333,0.577) -- (0.167,0.866) -- (0.333,1.155) -- (0.167,1.443) -- (-0.167,1.443) -- (-0.333,1.155) -- (-0.167,0.866) -- (-0.333,0.577) -- (-0.667,0.577) -- (-0.833,0.287) -- (-0.667,0)-- (-0.333,0) -- (-0.167,0.287) -- (0.167,0.287);
		
		\draw[thick] (0.167,0.287) -- (0.333,0.577);
		\draw[thick] (0.167,0.866) -- (-0.167,0.866);
		\draw[thick] (-0.167,0.287) -- (-0.333,0.577);
		
		\node at (0,0.6) {$1$};
		\node at (0,0.88) {$2$};
		\node at (0,1.15) {$3$};
		\node at (0,1.52) {$4$};
		\node at (-0.35,1.3) {$5$};
		\node at (0.35,1.3) {$6$};
		\node at (-0.35,1) {$7$};
		\node at (0.35,1) {$8$};
		
	\end{tikzpicture}
	\caption{Seven compact rational cells in $\underline{\M}^{comb}_{0,3}$ labeled by numbers.}
	\label{modulispheregreen}
\end{figure}
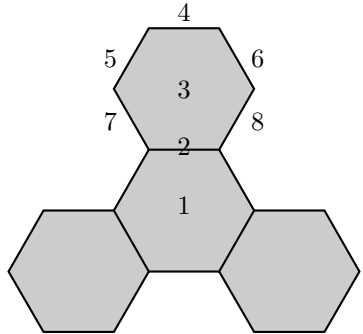

\begin{figure}
	\centering
    \includegraphics{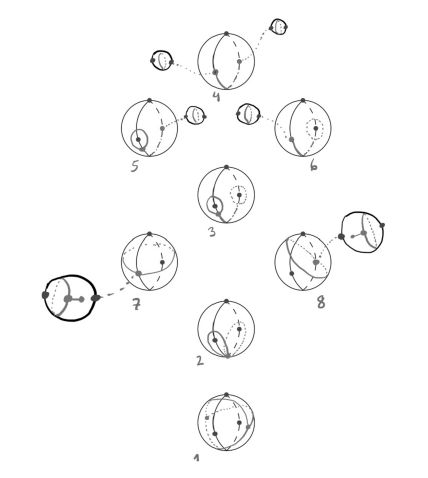}
	\caption{The corresponding seven critical graphs appear inside Riemann spheres with three marked points in a lighter shade of grey.}
\label{criticalgraphsgreen}
\end{figure}

Let us define $\underline{\varphi} : \M^{dec}_{g,n} \to \underline{\M}^{comb}_{g,n}$ by simply changing the codomain of $\varphi$ and notice that the closure of the image of this new map is $\underline{\M}^{comb}_{g,n}$. The seven compact rational cells numbered in Figure~\ref{modulispheregreen} can be represented via $\underline{\varphi}$ by the seven critical graphs numbered in Figure~\ref{criticalgraphsgreen}. This also illustrates the deformations of metric ribbon graphs. The critical trajectories of the sufaces labeled as 4, 5, 6, 7 and 8 correspond with generalizations of the concept of metric ribbon graph. For example, let $a,b \in E(\Gamma)$ be the loops of the double noose in 3. As these loops contract, that is, their lengths $l(a)$ and $l(b)$ go to zero, the ribbon graph degenerates to an interval (no longer a proper metric ribbon graph). However, due to the blow-up, those two lengths are not lost: they form a one dimensional family represented by the two smaller surfaces with two marked points that are \emph{attached} to the degenerated ribbon graph. The one parameter famliy is $\Delta_{a,b}$.

For a bordered Riemann surface of type $(g,b,n,\vec{m})$, let $c_i \ge 0$ with $1 \le i \le n$ denote decorations for the marked points in the interior of the surface and let $o_j \ge 0$ with $1 \le j \le m$ denote decorations for the marked points in the boundary of the surface. Suppose that $2 \sum c_i + \sum o_j =1$. All these decorations can be considered in the interior of the $(n+m-1)$-dimensional simplex $int(\Delta^{n+m-1})$. 

\begin{definition}
The decorated moduli space of bordered surfaces of type $(g,b,n,\vec{m})$ is defined as \[ \M^{dec}_{g,b,n,\vec{m}} = \M_{g,b,n,\vec{m}} \times int(\Delta^{n+m-1}).\] The decorated moduli space of symmetric surfaces of type $(g,b,k)$ and $(n,m)$ marked points is defined as \[ \R \M^{dec}_{(g,b,k),(n,m)} = \R \M_{(g,b,k),(n,m)} \times int( \Delta^{n+m-1} ) \] where we require that the decorations of conjugate marked points be equal. 
\end{definition}

By doubling the decorations of the $n$ interior marked points, we obtain a composition \[ \M^{dec}_{g,b,n,\vec{m}} \stackrel{d}{\longrightarrow} \R \M^{dec}_{(2g+b-1, b, 0), (n, m)} \stackrel{f}{\longrightarrow} \M^{dec}_{2g+b-1, 2n +m}\]  with similar characteristics as before. 

\begin{definition}
Let ${\mathcal S}^{dec}_{g,b,n,\vec{m}} = f \circ d ( \M^{dec}_{g,b,n,\vec{m}} )$. The moduli space of symmetric metric ribbon graphs is defined by \[ {\mathcal S}^{comb}_{g,b,n,\vec{m}} = \varphi ({\mathcal S}^{dec}_{g,b,n,\vec{m}})\] and \[ \ul{{\mathcal S}}^{comb}_{g,b,n,\vec{m}} = cl( \ul{\varphi} ({\mathcal S}^{dec}_{g,b,n,\vec{m}}) ) \] is the moduli space of stable symmetric metric ribbon graphs.
\end{definition}

\begin{theorem}
The combinatorial moduli space $\underline{\mathcal S}^{comb}_{g,b,n,\vec{m}}$ is a compact Hausdorff space and it is a compactification of the decorated moduli space ${\mathcal S}^{dec}_{g,b,n,\vec{m}}$.
\end{theorem}

\begin{proof}
The combinatorial moduli space $\underline{\mathcal S}^{comb}_{2g+b-1,2n+m}$ is a compact, Hausdorff space since it is a finite union of compact rational cells. This implies that $\underline{\mathcal S}^{comb}_{g,b,n,\vec{m}}$ is a compact Hausdorff space. Moreover, since \[ \underline{\mathcal S}^{comb}_{2g+b-1,2n+m} \setminus {\mathcal S}^{comb}_{2g+b-1,2n+m} \] is made of lower dimensional rational faces and $\ul{\varphi}$ is a homeomorphism, then the decorated space ${\mathcal S}^{dec}_{g,b,n,\vec{m}}$ compactifies to $\underline{\mathcal S}^{comb}_{g,b,n,\vec{m}}$
\end{proof}
 
 The real dimension of this moduli space is $6g+3b+3n+2m-7$. This follow by counting dimensions of the bordered moduli space and the simplex defining the decorations. The lowest dimensional examples with boundary and negative Euler characteristic are shown in Table~\ref{lowdimtab}.
 
 \begin{table}
 \caption{Euler characteristic of a bordered surfaces of the corresponding topological type and dimension of the corresponding combinatorial moduli space.}
 \label{lowdimtab}
 $\begin{array}{cccc|c|c}
g & b & n & m & \chi & \dim \\
\hline
%0 & 0 & 3 & 0 & -1 & 2 \\
%1 & 0 & 1 & 0 & -1 & 2 \\
\hline  
0 & 1 & 0 & 3 & -1/2 & 2 \\ 
0 & 1 & 1 & 1 & -1/2 & 1 \\ 
0 & 1 & 1 & 2 & -1 & 3 \\  
0 & 1 & 2 & 0 & -1 & 2 \\ 
\hline
0 & 2 & 0 & 1 & -1/2 & 1 \\
0 & 2 & 0 & 2 & -1 & 3 \\
0 & 2 & 1 & 0 & -1 & 2 \\
\end{array}$
\end{table}

It would be interesting to define $\underline{\mathcal S}^{comb}_{g,b,n,\vec{m}}$ in a purely combinatorial way. That would involve a combinatorial notion of bordered ribbon graphs and their semistable versions.

\section{A few examples}

\begin{figure}
\centering
\begin{tikzpicture}[scale=1.8]
	
	\draw[thin, dashed, fill=light-gray] (-1,0) -- (0,1.732) -- (1,0) -- (-1,0);
	\draw[thin] (0,0) -- (-0.5,0.866) -- (0.5,0.866) -- (0,0);
	
	\draw[line width=0.5mm] (0,0) -- (0,1.732);
	
	\draw[fill=white] (0,0) circle (1.2pt);
	\draw[fill=white] (-1,0) circle (1.2pt);
	\draw[fill=white] (1,0) circle (1.2pt);
	\draw[fill=white] (-0.5,0.866) circle (1.2pt);
	\draw[fill=white] (0.5,0.866) circle (1.2pt);
	\draw[fill=white] (0,1.732) circle (1.2pt);
	
\end{tikzpicture}
\hspace{2cm}
\begin{tikzpicture}[scale=2]
	
	\draw[fill=light-gray, thin] (0.167,0.287) -- (0.333,0) -- (0.667,0) -- (0.833,0.287) -- (0.667,0.577) -- (0.333,0.577) -- (0.167,0.866) -- (0.333,1.155) -- (0.167,1.443) -- (-0.167,1.443) -- (-0.333,1.155) -- (-0.167,0.866) -- (-0.333,0.577) -- (-0.667,0.577) -- (-0.833,0.287) -- (-0.667,0)-- (-0.333,0) -- (-0.167,0.287) -- (0.167,0.287);
	
	\draw[thin] (0.167,0.287) -- (0.333,0.577);
	\draw[thin] (0.167,0.866) -- (-0.167,0.866);
	\draw[thin] (-0.167,0.287) -- (-0.333,0.577);
	
	\draw[very thick] (0,0.289) -- (0,1.443);
	
	\fill[black] (0,0.289) circle (1pt);
	\fill[black] (0,0.866) circle (1pt);
	\fill[black] (0,1.443) circle (1pt);
	
\end{tikzpicture}
\caption{The spaces ${\mathcal S}^{comb}_{0,1,1,1}$ and $\underline{\mathcal S}^{comb}_{0,1,1,1}$ are depicted in the middle with thicker lines.}
\label{modulisphere}
\end{figure}
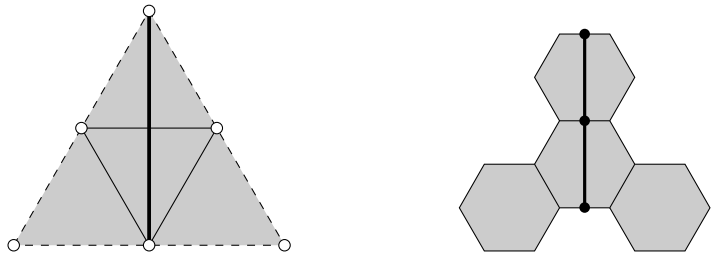

The moduli space ${\mathcal S}^{comb}_{0,1,1,1}$ sits inside $\M^{comb}_{0,3}$ and must be one dimensional. Considering symmetric versions of the theta graph, figure eight and double noose, Figure~\ref{modulisphere} shows on the left ${\mathcal S}^{comb}_{0,1,1,1}$ represented by the thicker subspace in the middle of the triangle. On the right of  Figure~\ref{modulisphere} the space $\underline{\mathcal S}^{comb}_{0,1,1,1}$ is depicted as the thicker subspace of $\underline{\M}^{comb}_{0,3}$. The compact combinatorial moduli space consists of three 0-dimensional cells and two 1-dimensional cells.

\begin{figure}
	\centering
\includegraphics{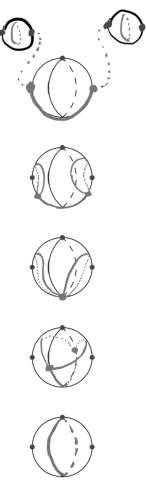}
\caption{Critical graphs corresponding with the compact rational cells of $\underline{\mathcal S}^{comb}_{0,1,1,1}$. The critical graphs appear inside the Riemann with three marked points spheres in a lighter shade of grey.}
\label{modulisphereone}
\end{figure}

Figure~\ref{modulisphereone} shows the five critical graphs corresponding to the five compact rational cells of the space $\underline{\mathcal S}^{dec}_{0,1,1,1}$.

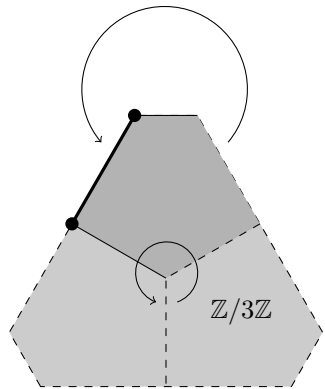
\begin{figure}
	\centering
	\begin{tikzpicture}[scale=2.5]
		
		\draw[thin, dashed, fill=light-gray] (-0.667,0) -- (-0.833,0.289) -- (-0.167,1.443) -- (0.167,1.443) -- (0.833,0.289) -- (0.667,0) --  (-0.667,0);
		
		\fill[color=gray!60] (-0.5,0.866) -- (-0.167,1.443) -- (0.167,1.443) -- (0.5,0.866) -- (0,0.577);
		\draw[thin] (0.167,1.443) -- (-0.167,1.443) -- (-0.5,0.866) -- (0,0.577);
		\draw[thin, dashed] (0.5,0.866) -- (0,0.577) -- (0,0);
		
		\draw[->] (0.06,0.45) arc [radius=0.165, start angle=-70, end angle= 250];
		\node at (0.4,0.4) {$\Z/ 3\Z$};
		
		\draw[thin, ->] (0.33,1.3) arc [radius=0.44, start angle=-40, end angle= 220];
		
		\draw[very thick] (-0.167,1.443) -- (-0.5,0.866);
		\fill[black] (-0.167,1.443) circle (1pt);
		\fill[black] (-0.5,0.866) circle (1pt);
		
	\end{tikzpicture}
	\caption{The space $\underline{\mathcal S}^{comb}_{0,2,0,1}$ is the thicker edge with two univalent vertices.}
	\label{modulicombtorus}
\end{figure}

\begin{figure}
	\centering
	\includegraphics{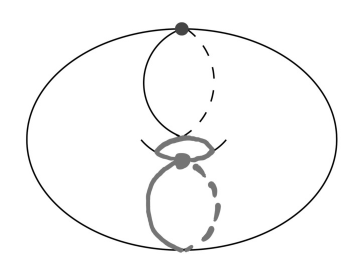}
	\caption{Critical graph in $\underline{\mathcal S}^{comb}_{0,2,0,1}$. The critical graph is the twisted figure eight on the bottom-center of the torus.}
	\label{criticalgraphtorusone}
\end{figure}

The space $\underline{\mathcal S}^{comb}_{0,2,0,1}$ is one dimensional and sits inside $\underline{\M}^{comb}_{1,1}$, this is depicted by the thicker line in Figure~\ref{modulicombtorus}. The moduli space $\underline{\M}^{comb}_{0,2,0,1}$ consists of two 0-dimensional compact rational cells and one 1-dimensional compact rational cell. A typical symmetric critical graph in $\underline{\M}^{comb}_{1,1}$ describing a point in $\underline{\mathcal S}^{comb}_{0,2,0,1}$ is shown in Figure~\ref{criticalgraphtorusone}.

\begin{figure}
	\centering
\includegraphics{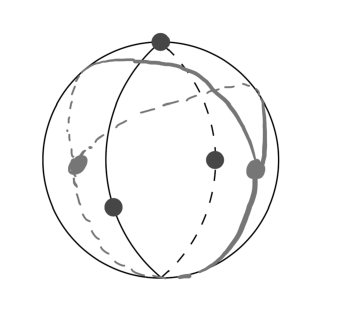}
	\caption{Critical graph representing a point in $\underline{\mathcal S}^{comb}_{0,1,0,3}$. The critical graph is the theta graph inside the sphere represented with a lighter shade of grey.}
\label{spherecritical}
\end{figure}

Figure~\ref{spherecritical} shows a theta graph as a symmetric critical trajectory in $\underline{\M}^{comb}_{0,3}$. In fact, all graphs of type $(0,3)$ can be placed in the double of the disc with three marked points on the boundary as symmetric graphs. This implies that $\underline{\M}^{comb}_{0,1,0,3} \cong \underline{\M}^{comb}_{0,3}$ whose picture can be found in Figure~\ref{compactmodulitheta}.

\begin{figure}
	\centering
	\includegraphics{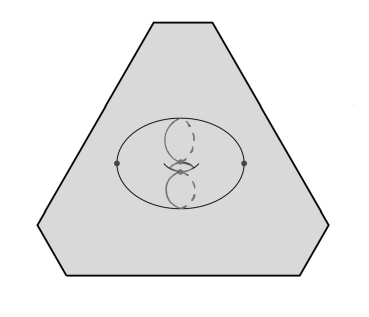}
	\caption{The space $\underline{\mathcal S}^{comb}_{0,2,1,0}$ and the critical graph representing the interior. The critical graph is represented by three loops on the middle of the torus in a lighter shade of grey.}
	\label{torustwopoints}
\end{figure}

$\underline{\mathcal S}^{comb}_{0,2,1,0}$ consists of only one 2-dimensional cell as shown in Figure~\ref{torustwopoints}. That figure also shows the double of annulus with one marked interior point, that is, the symmetric torus with two marked points. In the torus we see the critical trajectory right in the middle which is a ribbon graph of type $(1,2)$.

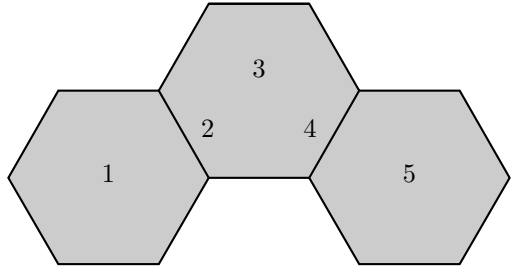
\begin{figure}
	\centering
\begin{tikzpicture}[scale=4]
	
	\draw[fill=light-gray, thick] (0.167,0.287) -- (0.333,0) -- (0.667,0) -- (0.833,0.287) -- (0.667,0.577) -- (0.333,0.577) -- (0.167,0.866) -- (-0.167,0.866) -- (-0.333,0.577) -- (-0.667,0.577) -- (-0.833,0.287) -- (-0.667,0)-- (-0.333,0) -- (-0.167,0.287) -- (0.167,0.287);
	
	\draw[thick] (0.167,0.287) -- (0.333,0.577);
	\draw[thick] (0.167,0.866) -- (-0.167,0.866);
	\draw[thick] (-0.167,0.287) -- (-0.333,0.577);
	
	\node at (-0.5,0.3) {1};
	\node at (-0.17,0.45) {2};
	\node at (0,0.65) {3};
	\node at (0.17,0.45) {4};
	\node at (0.5,0.3) {5};
	
\end{tikzpicture}
	\caption{The space $\underline{\mathcal S}^{comb}_{0,1,2,0}$ The numbers label different racional cells.}
\label{spherefourpoints}
\end{figure}

\begin{figure}
	\centering
	\includegraphics{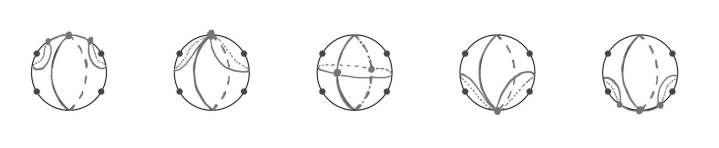} \par
	\vspace{-0.5cm}
	{1 \hspace{2.05cm} 2 \hspace{2.05cm} 3 \hspace{2.05cm} 4 \hspace{2.05cm} 5}
	\caption{Critical trajectories in a symmetric sphere with four marked points. The numbers correspond with the rational cells in Figure~\ref{spherefourpoints}.}
	\label{spherefourpointstra}
\end{figure}

The space $\underline{\mathcal S}^{comb}_{0,1,2,0}$ is formed by three 2-dimensional cells as depicted in Figure~\ref{spherefourpoints}.  In that figure the 2-dimensional cells are labeled together with two additional 1-dimensional cells. They correspond with the critical graphs in Figure~\ref{spherefourpointstra}. The double of the disc with two interior marked points is the symmetric Riemann sphere with four marked points.

\section{Relations to known spaces}

Ideally, we would like to construct orbicell decompositions for the moduli spaces $\R \M^{dec}_{(2g+b-1,b,0),(n,m)}$. However since the forgetful map $f$ is not injective in general, it is not possible to do so in a straightforward manner. For $g=0$ and $b=1$, the forgetful map is injective since the moduli spaces for Riemann spheres are fine (see \cite{ceyhan2006}). The genus one examples in the previous section are also decompositions for the corresponding real moduli spaces because the real and complex automorphism group of an elliptic curve are equal, thus the injectivity of the forgetful map follows (see \cite{silhol1992}). 

A possible way forward would be to study the difference between complex and real automorphisms of a real curve. As suggested by the referee, one can consider the cone of the morphism \[ \Psi_* : C_\bullet (\M_{g,b,n,\vec{m}}, \Q) \to C_\bullet (\M_{2g+b-1, 2n+m} ,\Q) \cong C_\bullet (\M^{comb}_{2g+b-1,2n+m}, \Q) \] The complex on the lefthand side admits a combinatorial model in terms of BW ribbon graphs. The complex on the right admits a model in terms of classical ribbon graphs. The corresponding cone can be considered as the combinatorial tool to detect the differences between complex and real isomorphisms of the double of a bordered surface.

There is also a connection between $\ul{S}^{comb}_{g,b,n,\vec{m}}$ and $\ul{\M}_{2g+b-1,2n+m}$. It should be possible to construct a non-surjective map \[ \pi : \ul{S}^{comb}_{g,b,n,\vec{m}} \to \ul{\M}_{2g+b-1,2n+m} \] which in the interior corresponds with $\varphi^{-1}$ and forgetting the decorations by real numbers. On the boundary, the construction of the stable surface can be achieved by considering sequences of nested subgraphs as in \cite{zun2015}. The decorations by real tangent directions at the nodes can be obtained by remembering how the subgraphs degenerate. The exact nature of this map and its fibers needs further investigation.

In \cite{cos:dpv} a combinatorial model $D_{g,b,m,n}$ is introduced formed by BW graphs as defined in \cite{WW2016}. This complex is weakly homotopic to a partial compactification of the moduli space $\M_{g,b,n,\vec{m}}$. It is interesting to note a relation between BW graphs and the graphs presented here. The first three examples in the previous section have marked points in the boundary and one can obtain a BW graph choosing half of one of the critical trajectories. 

\begin{figure}
	\centering
\begin{tikzpicture}[scale=1.2]
	
\draw[very thick] (-1,-0.15) -- (-1,0.65);
\draw[fill=white, thick] (-1,0.65) circle (2.5pt);

\draw[very thick] (2,0.5) circle (0.5cm);
\draw[fill=black] (2,0) circle (2.5pt);
\draw[very thick] (2,0) -- (2,-0.5);

\draw[fill=black] (5,0) circle (2.5pt);
\draw[very thick] (5,0) -- (5,1);
\draw[very thick] (5,0) -- (5.866,-0.5);
\draw[very thick] (5,0) -- (4.134,-0.5);
	
\end{tikzpicture}
	\caption{BW graphs in $D_{0,1,1,1}$, $D_{0,2,1,0}$, and $D_{0,1,3,0}$ respectively.}
\label{bwgraphs}
\end{figure}
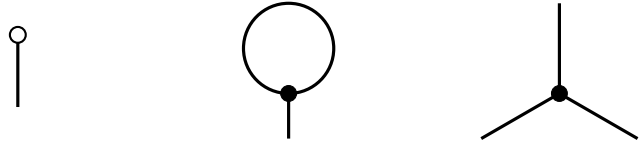

At the top of Figure~\ref{modulisphereone} there is a graph consisting of an edge with two vertices and two extra circles. One half of the single edge (say the left half) gives the first BW graph in Figure~\ref{bwgraphs} with only one white vertex corresponding with the interior marked point and one half-edge corresponding to the boundary marked point. Notice that this graph only has one boundary cycle since $b=1$. Figure~\ref{criticalgraphtorusone} shows the twisted figure eight as a critical graph. One half of this graph (the left half) gives the second BW graph on Figure~\ref{bwgraphs} with two boundary cycles since $b=2$, one black trivalent vertex, and one half-edge since $m=1$. There are no white vertices since $n=0$. Finally, Figure~\ref{spherecritical} shows the theta graph as a critical graph. Again, one half of this graph (the left half again) gives the third BW graph in Figure~\ref{bwgraphs} with only one boundary cycle since $b=1$, three half-edges since $m=1$, and no white vertex since $n=0$. For the last two examples $m=0$ and the same idea can not be applied. However, for the critical trajectory in Figure~\ref{torustwopoints}, after contracting the middle loop, there is a dual graph in the left half of the torus: a loop with one edge and one white vertex at the interior marked points. This BW graph has two boundary cycles since $b=2$. For the critical graph marked as ``3'' in Figure~\ref{spherefourpointstra}, after contracting the middle horizontal loop, there is a dual BW graph in the left half with two white vertices at the interior marked points joined by a single edge. This seems to indicate a way to extract BW graphs in $D_{g,b,m,n}$ from $\ul{S}^{comb}_{g,b,n,\vec{m}}$ in a similar way in which one chooses a bordered Riemann surface out of a symmetric one.

\section{Acknowledgment}

We would like to thank the referee for pointing out several shortcomings of the initial version of this article and also for suggesting the way for improving it. His contributions made this work more precise, organized and readable. 

\bibliographystyle{amsalpha}
\bibliography{cmbs}

\end{document}